\documentclass[11 pt]{amsart}
\usepackage{amsmath}
\usepackage{amsfonts}
\usepackage{amssymb}
\usepackage{graphicx}
\usepackage{float, verbatim}
\usepackage{enumerate}
\usepackage{color}
\usepackage{hyperref}

\numberwithin{equation}{section}

\newcommand{\Assouad}{\dim_{\mathrm{A}}}

\renewcommand{\epsilon}{\varepsilon}

\newtheorem{theorem}{Theorem}[section]

\newtheorem{question}[theorem]{Question}

\begin{document}

\title[Asymptotically and omnidirectionally containing APs]{Construction of a one-dimensional set which asymptotically and omnidirectionally contains arithmetic progressions}

\author[K. Saito]{ Kota Saito }
\address{Kota Saito\\
Graduate School of Mathematics\\ Nagoya University\\ Furocho\\ Chikusa-ku\\ Nagoya\\ 464-8602\\ Japan }
\curraddr{}
\email{m17013b@math.nagoya-u.ac.jp}
\thanks{This paper will appear at \textit{Journal of Fractal Geometry}.}


\maketitle

\begin{abstract}
 In this paper, we construct a subset of $\mathbb{R}^d$ which asymptotically and omnidirectionally contains arithmetic progressions but has Assouad dimension 1. More precisely, we say that $F$ asymptotically and omnidirectionally contains arithmetic progressions if we can find an arithmetic progression of length $k$ and gap length $\Delta>0$ with direction $e\in S^{d-1}$ inside the $\epsilon \Delta$ neighbourhood of $F$ for all $\epsilon>0$, $k\geq 3$ and $e\in S^{d-1}$. Moreover, the dimension of our constructed example is the lowest-possible because we prove that a subset of $\mathbb{R}^d$ which asymptotically and omnidirectionally contains arithmetic progressions must have Assouad dimension greater than or equal to 1. We also get the same results for arithmetic patches, which are the higher dimensional extension of arithmetic progressions.  
\end{abstract}

\section{Introduction}

Let $\{e_1, \dots, e_m\}$ be a set of linearly independent vectors in $\mathbb{R}^d$ where $1\leq m \leq d$, $m,d\in \mathbb{N}$. For $k\in \mathbb{N}$ and $\Delta>0$ we say that a set $P\subset \mathbb{R}^d$ is an {\it arithmetic patch (AP) } of size $k$ and scale $\Delta$ with respect to orientation $\{e_1,\ldots,e_m\}$ if 
\[
	P=\left\{t+\Delta\sum_{i=1}^m x_ie_i \ \colon \ x_i=0,1,\ldots,k-1 \right\}
\] 
for some $t\in \mathbb{R}^d$. We say that this $t$ is the initial point of $P$. For $\epsilon\in [0,1)$, we say that $Q \subset \mathbb{R}^d$ is a $(k, \varepsilon, \{e_1,\ldots,e_m\})${\it -AP} if there exists an arithmetic patch $P$ of size $k$, and scale $\Delta >  0$ with respect to orientation $\{e_1,\ldots,e_m\}$, such that
\begin{equation}\label{f1.1}
\# Q=\# P 
\end{equation}
and
\begin{equation}\label{f1.2}
\inf_{y \in Q} \|x-y\| \leq \epsilon \Delta 
\end{equation}
for all $x\in P$, where $\|\cdot\|$ denotes the Euclidean norm on $\mathbb{R}^d$. If $t$ is the initial point of $P$, we can take $t'\in Q$ such that $\|t-t'\|\leq \epsilon \Delta$. This $t'$ is uniquely determined if $\varepsilon>0$ is sufficiently small. Then this $t'$ is called the initial point of $Q$. 

Note that a $(k,0,\{e_1,\dots,e_m\})$-AP is an arithmetic patch of size $k$ with orientation $\{e_1,\ldots,e_m\}$. This notion was first introduced by Fraser and Yu \cite{FraserYu}, and refined by Fraser, the author and Yu \cite{FraserSaitoYu}. We say that $F\subset \mathbb{R}^d$ \emph{asymptotically and omnidirectionally contains arithmetic progressions} if for all $k\geq 3, \epsilon\in(0,1) $ and $e\in S^{d-1}$ there exists a $(k,\epsilon, \{e\})$-AP $Q\subset \mathbb{R}^d$ such that $Q\subseteq F$, where $S^{d-1}$ denotes the unit sphere in $\mathbb{R}^d$. \\

Fraser, the author and Yu asked the following question \cite[Question 6.2]{FraserSaitoYu}:

\begin{question}[a discrete analogue of the Kakeya problem]\label{Ques1}
Let $\epsilon \in (0,1)$ and suppose $F \subseteq \mathbb{R}^d$ contains a $(k, \varepsilon, \{e\})$-AP for every direction $e \in S^{d-1}$ and every $k \geq 3$.  Is it true that the Assouad dimension of $F$ is necessarily equal to $d$?
\end{question}

Additionally we discuss the following question:
 
\begin{question}\label{Ques2}
Suppose $F \subseteq \mathbb{R}^d$ asymptotically and omnidirectionally contains arithmetic progressions. Is it true that the Assouad dimension of $F$ is necessarily equal to $d$?
\end{question}

The goal of this paper is to prove the negative answer to Question \ref{Ques1} and Question \ref{Ques2}. To achieve this goal,  in Section \ref{proof2} we will construct a subset of $\mathbb{R}^d$ which asymptotically and omnidirectionally contains arithmetic progressions but has Assouad dimension 1. Moreover we will prove that the dimension of our constructed example is the lowest-possible in Section  \ref{proof1}. We will briefly recall the definition of Assouad dimension in Section \ref{proof1}, but refer the reader to \cite{Fraser, Robinson} for more details. It is well-known that the Assouad dimension is always an upper bound for the Hausdorff dimension, and (for bounded sets) the upper box dimension. We refer the reader to \cite{Falconer} for more background on Hausdorff and box dimensions.\\

From \cite[Theorem 2.4]{FraserYu} it immediately follows that if $F\subseteq\mathbb{R}^d$ has Assouad dimension $d$, then $F$ asymptotically and omnidirectionally contains arithmetic progressions. Thus the affirmative answer to Question \ref{Ques2} would imply the converse of this statement. In addition, this question is related to the Kakeya problem which is to prove that if a compact set $K \subseteq \mathbb{R}^d$ contains a unit line segment in every direction then it necessarily has Hausdorff dimension $d$. The affirmative answer to Question \ref{Ques2} would imply that every Kakeya set has Assouad dimension $d$. Kakeya-type problems are connected to many other problems (see \cite{Tao}). 

 In the case $d=1$, the answer to Question \ref{Ques2} is ``yes'' from \cite[Theorem 2.4]{FraserYu}. Surprisingly we prove that the answer to Question \ref{Ques2} is ``no'' in the case $d\geq 2$.\\

\section{Main results}\label{main}

\begin{theorem}\label{Th1}
Let $d\geq 2$ be an integer. There exists a set $K\subset \mathbb{R}^d$ such that 
\begin{itemize}
\item[(i)] $K$ is compact,
\item[(ii)] $K$ asymptotically and omnidirectionally contains arithmetic progressions,
\item[(iii)] $\Assouad K=1$.
\end{itemize}

\end{theorem}
Here $\Assouad K$ denotes the Assouad dimension of $K$, whose precise definition will be given in Section \ref{proof2}. We will prove this theorem in Section \ref{proof2}.\\

Let $m,d$ be integers such that $1\leq m\leq d$, and let 
\[
	B^d_m=\big\{\{e_1,\ldots,e_m\}\ \colon\ \{e_1,\ldots, e_m\}\subset \mathbb{R}^d\text{  is linearly independent} \big\}.
\]    
For every $\{e_1,\ldots,e_m\}\in B^d_m$, we define
\[
	\ell(\{e_1,\ldots,e_m\})=\min_{\omega\in\{0,1\}^m\setminus \{0\}^m}\|\omega_1e_1+\cdots+\omega_m e_m\|\:,
\]
and
\[
L(\{e_1,\ldots,e_m\})=\max_{\omega\in\{0,1\}^m}\|\omega_1e_1+\cdots+\omega_m e_m\|\:
\]
where $\omega=(\omega_1,\ldots,\omega_m)$.

\begin{theorem}\label{Th2}
Fix $\{e_1\ldots,e_m\} \in B^d_m$ and $0\leq \epsilon <\ell(\{e_1,\ldots,e_m\})/2$. If $F\subseteq \mathbb{R}^d$ contains $(k,\epsilon, \{e_1,\ldots,e_m\})$-APs for all $k\geq 3$, then $\Assouad  F\geq m$.
\end{theorem}

From Theorem, \ref{Th2}, it follows that there does not exist a set which satisfies (i),(ii) in Theorem \ref{Th1} and has Assouad dimension strictly less than 1. Therefore the constructed example of Theorem \ref{Th1} is of the lowest-possible dimension.

\section{Proofs }\label{proof1}\label{proof2}
We now briefly recall the definition of Assouad dimension. For every bounded set $E \subset \mathbb{R}^d$ $(d \in \mathbb{N})$ and $r>0$, let $N (E, r)$ be the smallest number of open sets with diameter less than or equal to $r$ required to cover $E$.  The \emph{Assouad dimension} of a set $F\subseteq \mathbb{R}^d$ is defined by
\begin{align*}
\Assouad F & =   \inf \Bigg\{ \  \sigma \geq 0 \  \colon \ \text{      $ (\exists \, C>0)$ $(\forall \, R>0)$ $(\forall \, r \in (0,R) )$ $(\forall \, x \in F )$ } \\ 
&\, \hspace{45mm} \text{ $N\big( B(x,R) \cap F , r\big) \ \leq \ C \bigg(\frac{R}{r}\bigg)^\sigma$ } \Bigg\},
\end{align*}
where $B(x,R)$ denotes the closed ball centered at $x$ with radius $R$.\\

\begin{proof}[Proof of Theorem \ref{Th2}]
Fix integers $m$ and $d$ with $1\leq m\leq d$. Let $E=\{e_1,\ldots,e_m\}$. For every $k\geq 3$ we can take a $(k,\epsilon,E)$-AP $Q_k$ of $F$. Let $P_k$ be a $(k,0,E)$-AP which is approximated by $Q_k$, that is, $P_k$ satisfies (\ref{f1.1}) and (\ref{f1.2}). Let $\Delta_k>0$ be the scale of $P_k$ and $t_k\in F$ be the initial point of $Q_k$. We define $R_k=(1+2\epsilon)k\Delta_k(L(E)+1) $ and $r_k=(\ell(E)-2\epsilon)\Delta_k/2$ for every $k\geq 3$.  The diameter of $Q_k$ is at most 
\[
	\Delta_k(k-1)L(E)+2\epsilon \Delta_k \leq R_k
\]
for all $k\geq 3$. Hence we obtain
\begin{equation}\label{Th2f1}
	N\big(Q_k\cap B(t_k, R_k),r_k\big)= N\big(Q_k,r_k\big) \\
\end{equation}
for every $k\geq 3$ since $Q_k\cap B(t_k, R_k)=Q_k$. We take an open cover $\mathcal{U}$ of $Q_k$ such that
$\sup_{x,y\in U} \|x-y\|\leq r_k$ for all $U\in \mathcal{U}$. Since $\inf_{x,y\in Q_k, x\neq y}\|x-y\|\geq (\ell(E)-2\epsilon)\Delta_k>r_k$, we find that $U$ contains at most one point of $Q_k$ for every $U\in \mathcal{U}$. Therefore we have
\begin{equation}\label{Th2f2}
	N\big(Q_k,r_k\big)\geq \# Q_k =k^m=C \left(\frac{R_k}{r_k}\right)^m 
\end{equation}
for all $k\geq 3$, where let
\[
	C=\left(\frac{\ell(E)-2\epsilon}{2(1+2\epsilon)(L(E)+1)}\right)^m.
\]
Fix any $\sigma>\Assouad F$. From the definition of Assouad dimension, (\ref{Th2f1}), and (\ref{Th2f2}), there exists some constant $C_\sigma>0$ such that
\[
	C \left(\frac{R_k}{r_k}\right)^m\leq N\big(Q_k\cap B(t_k, R_k),r_k\big)\leq N(F\cap B(t_k,R_k),r_k)\leq C_{\sigma} \left(\frac{R_k}{r_k}\right)^\sigma
\]
for all $k\geq 3$. Since we find that $R_k/r_k \rightarrow \infty$ as $k\rightarrow\infty$, we have
$\sigma \geq m$. Therefore it follows that $\Assouad F\geq m$ by taking $\sigma\rightarrow \Assouad F +0$.  
\end{proof}

\begin{proof}[Proof of Theorem \ref{Th1}]
Fix an integer $d\geq 2$. Let $S'=\big\{x/\|x\|\: \colon \: x\in\mathbb{Q}^d\setminus\{0\} \big\}$. We can write $S'=\{\xi_0, \xi_1,\ldots\}$ since $S'$ is countable. Let $e_1=(1,0,\ldots, 0)\in \mathbb{R}^d$. For every $j=0,1,2,\ldots$, we define
\[
	T_j=\left\{t\xi_j +\frac{1}{2^j} e_1\ \colon\ t\in \left[-\frac{1}{2^{j+2}},\frac{1}{2^{j+2}}\right] \right\},
\] 
and
\[
	T_\infty =\{0\}\;.
\]
Note that $T_j$'s are line segments of length $1/2^{j+1}$ with direction $\xi_j$, and $T_j$'s do not have any intersections with each other. Let
\[
	K=\bigcup_{j\in \mathbb{N}_{0}^{*}} T_j 
\]
where $\mathbb{N}_{0}^{*}=\{0,1,2\ldots\}\cup \{\infty\} $. We will show that $K$ satisfies desired properties (i),(ii), and (iii). \\
	
	Firstly we show that $K$ satisfies (i) $i.e.$ $K$ is compact. If we take any sequence $\{a_j\}_{j=1}^\infty \subset K$, then either $\{a_j\}_{j=1}^\infty$ is contained a finite union of line segments or it travels through $T_{j_k}$
for infinitely many $j_k$. In the first case, there exists a subsequence $\{a_{j_\ell}\}_{\ell=1}^\infty$ which is convergent in $K$ from the  finiteness of the union. In the second case, there exists a subsequence $\{a_{j_\ell}\}_{\ell=1}^\infty$ converging to 0 from the definition of $K$. Therefore $K$ is compact. \\

	We next show that $K$ satisfies (ii) $i.e.$  $K$ contains $(k,\epsilon, \{e\})$-APs for all $k\geq 3$, $\epsilon\in(0,1)$, and $e\in S^{d-1}$. Fix $k\geq 3$ and $\epsilon\in(0,1)$. For every $e\in S^{d-1}$ we can take $\xi_{j_0}\in S'$ such that
\[
	\left\|e-\xi_{j_0} \right\|<\frac{2\epsilon}{k-1}\:
\]
since $S'$ is dense in $S^{d-1}$.
Here we define 
\[
	P_i=\left(-\frac{1}{2^{j_0+2}}+\frac{i}{(k-1)2^{j_0+1}}\right)e+\frac{1}{2^{j_0}}e_1 
\]
for all $i= 0,1,\ldots, k-1$. This sequence is a $(k,0,\{e\})$-AP of scale $1/((k-1)2^{j_0+1})$. We also define
\[
	Q_i=\left(-\frac{1}{2^{j_0+2}}+\frac{i}{(k-1)2^{j_0+1}}\right)\xi_{j_0}+\frac{1}{2^{j_0}}e_1 
\]
for all $i= 0,1,\ldots, k-1$. Note that $Q_i$'s are points in $T_{j_0}$ which partition $T_{j_0}$ into $k-1$ small line segments of same length. Let $\Delta=  1/((k-1)2^{j_0+1})$ which is the scale of $\{P_i\}_{i=0}^{k-1}$. Then we see that
\[
	\|P_i-Q_i\| \leq \left|-\frac{1}{2^{j_0+2}}+\frac{i}{(k-1)2^{j_0+1}}\right| \|e-\xi_{j_0}\| \leq \epsilon \Delta
\]
for all $i=0,1,\ldots, k-1$. Therefore $\{Q_i\}_{i=0}^{k-1}$ is a $(k,\epsilon, \{e\})$-AP, which proves the desired property.\\

	We last show that $K$ satisfies (iii) $i.e.$ $\Assouad K=1$. It is clear that $1\leq \Assouad K$ from Theorem \ref{Th2}. We now prove $\Assouad K\leq 1$.

Let $x\in K$. Then there exists $j\in \mathbb{N}^{*}_{0}$ such that $x\in T_{j}$. If $j=\infty$, then $x=0$ from the definition of $K$. Let $n(R)=\lfloor - \log R /\log 2 \rfloor $ for all $R>0$, where $\lfloor y\rfloor$ denotes the greatest integer less than or equal to $y$. Note that $1/2^{n(R)+1}<R\leq 1/2^{n(R)} $ for all $R>0$. Let $n^{+}(R)=\max\{0, n(R)\}$. Then we have
\begin{align*}
N\big(B(x,R)\cap K,r\big) &= N\big(B(0,R)\cap K,r\big)\\
&\leq N\big( T_{n^{+}(R)} \cup T_{n^{+}(R)+1} \cup \cdots \cup T_{\infty},r\big)\\
&\leq N\big(T_{\infty} , r)+\sum_{\ i\geq n^{+}(R)} N\big( T_i, r \big)\\
&\ll 1+\sum_{\ i\geq n^{+}(R)} \frac{1}{r2^{i+1}}  
\end{align*}
for all $0<r<R$, where for non-negative functions $f(R,r)$ and $g(R,r)$, we write $f(R,r)\ll g(R,r)$ for all $0<r<R$ if there exists some constant $C$ such that $f(R,r)\leq C g(R,r)$ for all $0<r<R$.
  Since $0\leq\ell(R)$ for all $ 0<R\leq 1$, we have
\[
	 1+\sum_{\ i\geq n^{+}(R)} \frac{1}{r2^{i+1}} \leq R/r +\frac{1}{r2^{n(R)}}\ll R/r
\]
for all $0<r<R\leq 1$. Since $n(R)<0$ for all $R>1$, we have
\[
 	1+\sum_{\ i\geq n^{+}(R)} \frac{1}{r2^{i+1}} =1+\sum_{i\geq 0} \frac{1}{r2^{i+1}}  \leq R/r +\frac{1}{r}\leq 2R/r
\]
for all $1<R$ and $0<r<R$. 
Therefore we obtain
\begin{equation}\label{f1Th4}
	N\big(B(x,R)\cap K,r\big)=N\big(B(0,R)\cap K,r\big) \ll R/r
\end{equation}
for all $0<r<R$.

If $j\neq \infty$ and $R<1/2^{j+3}$, then we see that $B(x,R)\cap K =B(x,R)\cap T_j$ from the definition of $K$. Therefore we have
\[
	N\big(B(x,R)\cap K,r\big)= N\big(B(x,R)\cap T_j,r\big) \leq N\big([-R,R],r\big)\ll R/r
\]   
for all $0<r<R<1/2^{j+3}$.
If $j\neq \infty$ and $R\geq1/2^{j+3}$, then we obtain
$B(x,R)\subset B\big(0,11R\big)$. In fact,
for all $y\in B(x,R)$ we have 
\[
	\|y\|\leq \|y-x\|+\|x-1/2^{j}e_1\|+\|1/2^{j}e_1\|\leq R+1/2^{j+2} +1/2^j\leq 11R  
\]
since $x\in T_j$ and $R\geq1/2^{j+3}$. Thus it follows that 
\[
	N\big(B(x,R)\cap K,r\big) \leq 
	N\big(B(0,11R)\cap K,r\big) 
\]
for all $1/2^{j+3}\leq R$ and $0<r<R$. From (\ref{f1Th4}), we have
\[
	N\big(B(0,11R)\cap K,r\big) \ll R/r\;
\] 
for all  $1/2^{j+3}\leq R$ and $0<r<R$. Therefore we have 
\[
	N\big(B(x,R)\cap K , r\big) \ll R/r
\]
for all $x\in K$ and $0<r<R$, which implies $\Assouad K \leq 1$ from the definition of Assouad dimension. 
\end{proof}

\section{The higher dimensional analogue}

We also get the higher dimensional result. 

\begin{theorem}\label{Th3}
Let $d\geq 2$ be an integer, and $1\leq m\leq d$ be an integer. There exists a set $K\subset \mathbb{R}^d$ such that 
\begin{itemize}
\item[(i)] $K$ is compact,
\item[(ii)] $K$ contains $(k,\epsilon, \{e_1,\ldots,e_m\})$-APs for all $k\geq 3$, $\epsilon\in(0,1)$, and \\ $\{e_1,\ldots, e_m\}\in B^{d}_m$,
\item[(iii)] $\Assouad K=m$.
\end{itemize}
\end{theorem}

Similarly from Theorem \ref{Th2}, there does not exist a set which satisfies (i),(ii) in Theorem \ref{Th3} and has Assouad dimension strictly less than $m$. Therefore the example of Theorem \ref{Th3} is of the lowest-possible dimension. 

\begin{proof}Fix integers $m$ and $d$ with $1\leq m\leq d$ and $d\geq 2$. Let $S'=\big\{x/\|x\|\:\colon\: x\in\mathbb{Q}^d\setminus\{0\} \big\}$. We can write
\[
	\underbrace{S'\times \cdots \times S'}_{m}
	 =\{{\xi}_0,{\xi}_1,\xi_2,\ldots\}
\]
since $\mathbb{Q}^d$ is countable. Let $\xi_j=(\xi_{1j},\xi_{2j},\cdots,\xi_{mj})$ for every $j\geq 0$. We define
\[
	D_j=\left\{ t\cdot \xi_j +\frac{1}{2^j}e_1\ 
	\colon\ t_1,\ldots ,t_m \in \left[-\frac{1}{m2^{j+2}}, 	\frac{1}{m2^{j+2}}\right] \right\}
\]
for every $j\geq 0$, where $t=(t_1,\ldots, t_m)$ and $t\cdot \xi_j=t_1\xi_{1j}+\cdots+t_m\xi_{mj}$, and we define
\[
	D_{\infty}=\{0\}.
\]
Note that $D_j$ are $m$-dimensional diamonds whose length of the one side is $1/(m2^{j+1})$. Then let
\[
	K=\bigcup_{j\in \mathbb{N}_0^{*}} D_j
\] 
where $ \mathbb{N}_0^{*}=\{0,1,2,\ldots\}\cup \{\infty\}$. $K$ satisfies (i),(ii) and (iii) in Theorem \ref{Th3}. This proof is similar to the proof of Theorem \ref{Th1}. Therefore we omit the proof.  
\end{proof}

\section{Future works}
 We ask some open questions related to the Kakeya problem. In the proof of Theorem \ref{Th1}, we construct ``countable'' line segments whose union asymptotically and omnidirectionally contains arithmetic progressions. This countability causes the lack of increase of dimension. Thus the answer to the following questions can be expected affirmatively.     

\begin{question}\label{Ques3}
Suppose a compact set $F\subseteq \mathbb{R}^d$ contains $(k,0, \{e\})$-APs for all $k\geq 3$ and $e\in S^{d-1}$. Is it true that the Assouad dimension of $F$ is necessarily equal to $d$? 
\end{question}

\begin{question}\label{Ques4}
Suppose a set $F\subseteq \mathbb{R}^d$ contains $(\infty,0, \{e\})$-APs for all $e\in S^{d-1}$.  
Is it true that the Assouad dimension of $F$ is necessarily equal to $d$?
\end{question}

The affirmative answer to Question \ref{Ques3} would imply the affirmative solution to the Kakeya problem for Assouad dimension. In the case $k=3$, Bourgain gave that a lower bound for the box dimension of a subset of $\mathbb{R}^d$ which contains $(3, 0, \{e\})$-APs for every direction $e \in S^{d-1}$ (see [B,Proposition 1.7]). 

\section*{Acknowledgements}
The author is particularly grateful for the comments given by Neal Bez, Jonathan M.~Fraser, Kohji Mastumoto, Han Yu, and the referee. The author is financially supported by Yoshida Scholarship Foundation.

\end{document}